\documentclass[12pt]{article}
\usepackage[a4paper,margin=1in]{geometry}

\usepackage{mathtools,amssymb,amsthm}
\mathtoolsset{showonlyrefs=true}

\newtheorem{theorem}{Theorem}[section]
\newtheorem{lemma}[theorem]{Lemma}
\newtheorem{corollary}[theorem]{Corollary}

\newtheorem{fact}[theorem]{Fact}

\theoremstyle{definition}
\newtheorem{definition}[theorem]{Definition}
\newtheorem{example}[theorem]{Example}

\theoremstyle{remark}

\usepackage{enumitem}
\setlist[enumerate,1]{label=(\arabic*)}
\setlist[enumerate,2]{label=(\roman*)}

\numberwithin{equation}{section}
\numberwithin{figure}{section}
\numberwithin{table}{section}

\usepackage{multirow}

\usepackage{hyperref}

\usepackage[style=alphabetic,isbn=false,url=false,doi=false,giveninits=true,maxnames=9,sorting=nyt,backend=bibtex]{biblatex} 
\renewbibmacro{in:}{}
\bibliography{reference.bib}

\title{Pushforward measures on homogeneous spaces of non-unimodular groups and properties of modular functions}
\author{Takashi Satomi}
\date{\textit{Dedicated to Professor Toshiyuki Kobayashi with gratitude to his continuous encouragement} \\[16pt]
\today} 

\begin{document}

\maketitle

\begin{abstract}

This paper shows that a formula for the pushforward measures on the fiber bundle $ G \to G / H $ of homogeneous space of locally compact groups $ G $ and $ H $ can be written down by using the modular functions.
As a result, we obtain the equality
\begin{align}
\frac{ \Delta_G ( h ) }{ \Delta_{ G / N } ( h N ) }
= \frac{ \Delta_H ( h ) }{ \Delta_{ H / N } ( h N ) }
\end{align}
on the modular functions for any closed subgroups $ N $ and $ H $ of a locally compact group $ G $ with $ N \lhd G $ and $ N \subset H \subset G $ and any $ h \in H $.

\end{abstract}



\section{Introduction}

We use the notation $ H < G $ and $ H \lhd G $ if $ H $ is a closed subgroup and a closed normal subgroup of a locally compact (Hausdorff) group $ G $, respectively.
This paper presents one can write down the pushforward measures on the fiber bundle $ p^{ G \to G / H } \colon G \to G / H $ of the homogeneous space of locally compact groups $ G $ and $ H $ with $ H < G $ by using the modular functions (Theorem \ref{thm:quotient-pushforward}).
As a result, we obtain an equality on the modular functions (Corollary \ref{cor:modular}).
That is, we have $ \Delta_G ( h ) / \Delta_{ G / N } ( h N ) = \Delta_H ( h ) / \Delta_{ H / N } ( h N ) $ for any $ N < H < G $ with $ N \lhd G $ and any $ h \in H $, where $ \Delta_G $ is the modular function of $ G $.
When $ H = N $ (i.e., $ H \lhd G $), this equality corresponds to a known result $ \Delta_G |_{ H } = \Delta_H $ (Example \ref{ex:normal-modular-function-equal}).

In this paper, we denote by $ C ( X ) $ the vector space of continuous functions $ X \to \mathbb{ C } $, and by $ C_c ( X ) \subset C ( X ) $ the the subspace of continuous functions with compact support.
A complex Radon measure on a locally compact Hausdorff space $ X $ is defined as a linear form of $ C_c ( X ) $.
When $ p \colon X_1 \to X_2 $ is a continuous map between locally compact Hausdorff spaces $ X_1 $ and $ X_2 $, the pushforward (whose precise definition will be given in Definition \ref{def:pushforward}) of the measure $ \mu \in C_c ( X_1 )^* $ is written as $ p_* ( \mu ) \in C_c ( X_2 )^* $.
If $ p $ is proper (i.e., the inverse image of any compact set is also compact), the pushforward $ p_* ( \mu ) $ can be defined for any Radon measure $ \mu \in C_c ( X_1 )^* $.
On the other hand, it needs the properness condition on the support of the measure $ \mu \in C_c ( X_1 )^* $ for the existance of the pushforward $ p_* ( \mu ) $ if $ p $ is not proper.

\begin{definition}

We denote by $ p^{ G \to G / H } \colon G \to G / H $ the canonical quotient map for locally compact groups $ G $ and $ H $ with $ H < G $.
The set $ M_H ( G ) \subset C_c ( G )^* $ of complex Radon measures is defined as
\begin{align}
    M_H ( G ) \coloneqq \{ \phi ( g ) dg \mid \phi \in C ( G ) , \; \text{ $ p^{ G \to G / H } |_{ \mathrm{ supp } ( \phi ) } $ is proper } \}.
\end{align}
Here $ \mathrm{ supp } ( \phi ) $ is the support of $ \phi $, and $ dg $ is a left Haar measure of $ G $ (but $ M_H ( G ) $ does not depend on the choice of $ dg $).

\end{definition}

\begin{example}
\label{ex:MH-H-trivial}

We have $ M_{ \{ e \} } ( G ) = \{ \phi ( g ) dg \mid \phi \in C ( G ) \} $, where $ e \in G $ is the identity element of $ G $.
    
\end{example}

The following theorem is one of the main theorems of this paper:

\begin{theorem}
\label{thm:main}

Let $ \overline{ G } \coloneqq G / N $ and $ \overline{ H } \coloneqq H / N $ for locally compact groups $ G $, $ H $, and $ N $ with $ N < H < G $.
We denote by $ p^{ G \to \overline{ G } } \colon G \to \overline{ G } $ and $ p^{ \overline{ G } \to G / H } \colon \overline{ G } \to G / H $ the canonical maps.

\begin{enumerate}
\item \label{item:main-pushforward-subset}

One has $ M_H ( G ) \subset M_N ( G ) $.

\item \label{item:main-pushforward-well-defined}

The pushforwards $ ( p^{ \overline{ G } \to G / H } \circ p^{ G \to \overline{ G } } )_* ( \mu ) $ and $ p_*^{ \overline{ G } \to G / H } \circ p_*^{ G \to \overline{ G } } ( \mu ) $ can be defined for any $ \mu \in M_H ( G ) $.
Furthermore, one has 
\begin{align}
    ( p^{ \overline{ G } \to G / H } \circ p^{ G \to \overline{ G } } )_* ( \mu ) = p_*^{ \overline{ G } \to G / H } \circ p_*^{ G \to \overline{ G } } ( \mu ). \label{eq:main-pushforward-well-defined-equal}
\end{align}

\item \label{item:main-normal-commutative}

 We assume $ N \lhd G $ and denote by $ p^{ G / H \to \overline{ G } / \overline{ H } } \colon G / H \to \overline{ G } / \overline{ H } $ and $ p^{ \overline{ G } \to \overline{ G } / \overline{ H } } \colon \overline{ G } \to \overline{ G } / \overline{ H } $ the canonical maps.
Then $ p_*^{ G / H \to \overline{ G } / \overline{ H } } \circ p_*^{ G \to G / H } ( \mu ) $ and $ p_*^{ \overline{ G } \to \overline{ G } / \overline{ H } } \circ p_*^{ G \to \overline{ G } } ( \mu ) $ can be defined.
Furthermore, one has 
\begin{align}
    p_*^{ G / H \to \overline{ G } / \overline{ H } } \circ p_*^{ G \to G / H } ( \mu ) = p_*^{ \overline{ G } \to \overline{ G } / \overline{ H } } \circ p_*^{ G \to \overline{ G } } ( \mu ). \label{eq:main-normal-commutative-state}
\end{align}

\item \label{item:main-normal-in-tilde}

For $ N \lhd G $, one has $ p_*^{ G \to \overline{ G } } ( \mu ) \in M_{ \overline{ H } } ( \overline{ G } ) $ for any $ \mu \in M_H ( G ) $.
    
\end{enumerate}
    
\end{theorem}

In addition, we can write down the pushforward measure of the fiber bundle $ p^{ G  \to G / H } $ by using the modular functions as follows:

\begin{theorem}
    \label{thm:quotient-pushforward}

Let $ G $, $ H $, $ p^{ G \to G / H } $ be as in Theorem \ref{thm:main}.
We write $ dg $ as a left Haar measure and $ \Delta_G $ as the modular function of $ G $.
Similarly, we write $ dh $ as a left Haar measure and $ \Delta_H $ as the modular function of $ H $.
Suppose $ \phi \in C ( G ) $ satisfies $ p^{ G \to G / H }|_{ \mathrm{ supp } ( \phi ) } $ is proper.
Then
\begin{align}
    \int_{ G / H } \int_{ H } \nu ( g h ) dh d p_*^{ G \to G / H } ( \phi ( g ) dg )
    = \int_{ G } \nu ( g ) \int_{ H } \phi ( g h ) \frac{ \Delta_G ( h ) }{ \Delta_H ( h ) } dh dg \label{eq:quotient-pushforward-state}
\end{align}
holds for any $ \nu \in C_c ( G ) $.
    
\end{theorem}

As a result, we obtain a property of the modular functions as follows:

\begin{corollary}
    \label{cor:modular}

Let $ G $, $ H $, $ N $, $ \overline{ G }$, and $ \overline{ H }$ be as in Theorem \ref{thm:main} with $ N \lhd G $.
Then
\begin{align}
    \frac{ \Delta_G ( h ) }{ \Delta_{ \overline{ G } } ( \overline{ h } ) }
    = \frac{ \Delta_H ( h ) }{ \Delta_{ \overline{ H } } ( \overline{ h } ) } \label{eq:modular-state}
\end{align}
holds for any $ h \in H $, where $ \overline{ h } \coloneqq p^{ G \to \overline{ G } } ( h ) $.

\end{corollary}

Corollary \ref{cor:modular} has been essentially known in the case of $ H = N $.
The proof of Corollary \ref{cor:modular} given later can be regarded as a generalization of the following example:

\begin{example}[Hewitt--Ross {\cite[Remark 15.23]{MR551496}}]
\label{ex:normal-modular-function-equal}

When $ H = N $ (in particular, $ H \lhd G $), we have $ \Delta_{ \overline{ G } } ( \overline{ h } ) = \Delta_{ \overline{ H } } ( \overline{ h } ) = 1 $.
Thus, Corollary \ref{cor:modular} implies $ \Delta_{ G } ( h ) = \Delta_{ H } ( h ) $ in this case.
    
\end{example}

Here is the organization of this paper.
In Section \ref{sec:pushforward}, we see the definition and properties of the pushforward measure, and we prove Theorem \ref{thm:main}.
In Section \ref{sec:proof-modular}, we prove Theorem \ref{thm:quotient-pushforward} and Corollary \ref{cor:modular} by using Theorem \ref{thm:main}.

\section{Pushforward measure}
\label{sec:pushforward}

In this section, we prove Theorem \ref{thm:main}.
In Subsection \ref{subsec:pushforward-definition}, we review the definition of the pushforward measure (Definition \ref{def:pushforward}) on locally compact Hausdorff spaces in a general setting and show several properties of this measure to prove Theorem \ref{thm:main}.
In Subsection \ref{subsec:pushforward-apply}, we apply these properties to the case of homogeneous spaces and complete the proof of Theorem \ref{thm:main}.

\subsection{Definition and properties of the pushforward measure}
\label{subsec:pushforward-definition}

In this subsection, we review the definition of the pushforward measure for locally compact Hausdorff spaces.
In addition, some properties of such a measure is also discussed.
For any continuous map $ p \colon X_1 \to X_2 $ between locally compact Hausdorff spaces $ X_1 $ and $ X_2 $, the set $ \mathcal{ M } ( p , X_1 ) $ is defined as
\begin{align}
   \mathcal{ M } ( p , X_1 )
    \coloneqq \{ \mu \in C_c ( X_1 )^* \mid \text{ $ p |_{ \mathrm{ supp } ( \mu ) }$ is proper } \}.
\end{align}
For any $ \mu \in\mathcal{ M } ( p , X_1 ) $ and $ \alpha_2 \in C_c ( X_2 ) $, there exists $ \alpha_1 \in C_c ( X_1 ) $ such that
\begin{align}
    \alpha_1 ( x ) = \alpha_2 \circ p ( x ) \label{eq:extention}
\end{align}
for any $x \in \mathrm{ supp } ( \mu ) $ by the Tietze extension theorem since $ \mathrm{ supp } ( \mu ) \cap \mathrm{ supp } ( \alpha_2 \circ p ) $ is compact.
In this case, the integration
\begin{align}
    \int_{ X_1 } \alpha_1 \; d \mu
\end{align}
does not depend on the choice of $ \alpha_1 $ by the definition of $ \mathrm{ supp } ( \mu ) $.
Thus, the pushforward measure is defined as follows:

\begin{definition}
    \label{def:pushforward}

Let $ p \colon X_1 \to X_2 $ be a continuous map between locally compact Hausdorff spaces $ X_1 $ and $ X_2 $.
Then the map $ p_* \colon\mathcal{ M } ( p , X_1 ) \to C_c ( X_2 )^* $ is defined as
\begin{align}
\int_{ X_2 } \alpha_2 \; d p_* ( \mu )
   \coloneqq \int_{ X_1 } \alpha_1 \; d \mu & & 
   (\mu \in \mathcal{ M } ( p , X_1 ), \; \alpha_2 \in C_c ( X_2 )),
\end{align}
where $ \alpha_1 \in C_c ( X_1 ) $ satisfies \eqref{eq:extention} for any $x \in \mathrm{ supp } ( \mu ) $.
This measure $ p_* ( \mu ) $ is called the pushforward of $ \mu $.

\end{definition}

The support of the pushforward measure is contained in the image of the original measure as follows:

\begin{lemma}
    \label{lem:pushforward-support}

Let $ X_1 $, $ X_2 $, and $ p $ be as in Definition \ref{def:pushforward}.
Then 
\begin{align}
    \mathrm{ supp } ( p_* ( \mu ) ) \subset p ( \mathrm{ supp } ( \mu ) ) \label{eq:pushforward-support-state}
\end{align}
holds for any $ \mu \in \mathcal{ M } ( p , X_1 ) $.
  
\end{lemma}

\begin{proof}

Let $ x_2 \in \mathrm{ supp } ( p_* ( \mu ) ) $.
Since $ p |_{ \mathrm{ supp } ( \mu ) } $ is a proper map to the locally compact Hausdorff space $ X_2 $ by $ \mu \in \mathcal{ M } ( p , X_1 ) $, the image $ p ( \mathrm{ supp } ( \mu ) ) $ of the closed set $ \mathrm{ supp } ( \mu ) $ is also closed.
Thus, it suffices to show $ p ( \mathrm{ supp } ( \mu ) ) \cap U \neq \emptyset $ for any open neighborhood $ U $ of $ x_2 $.
Since there exists $ \alpha_2 \in C_c ( X_2 ) $ such that $ \mathrm{ supp } ( \alpha_2 ) \subset U $ and
\begin{align}
    \int_{ X_2 } \alpha_2 \; d p_* ( \mu ) \neq 0
\end{align}
by $ x_2 \in \mathrm{ supp } ( p_* ( \mu ) ) $, there exists $\alpha_1 \in C_c ( X_1 ) $ such that
\begin{align}
    \int_{ X_1 } \alpha_1 \; d \mu \neq 0
\end{align}
and \eqref{eq:extention} hold for any $ x \in \mathrm{ supp } ( \mu ) $.
Thus, there exists $ x_1 \in \mathrm{ supp } ( \mu ) $ such that $ \alpha_1 ( x_1 ) \neq 0 $ and hence $ \alpha_2 \circ p ( x_1 ) \neq 0 $.
Then
\begin{align}
    p ( x_1 ) \in p ( \mathrm{ supp } ( \mu ) ) \cap \mathrm{ supp } ( \alpha_2 ) \subset p ( \mathrm{ supp } ( \mu ) ) \cap U
\end{align}
holds and hence we obtain \eqref{eq:pushforward-support-state}.
\end{proof}

Now we show the following lemma to prove Theorem \ref{thm:main}:

\begin{lemma}
    \label{lem:pushforward-compose}

Let $ p^{ 1 , 2 } \colon X_1 \to X_2 $ and $ p^{ 2 , 3 } \colon X_2 \to X_3 $ be continuous maps between locally compact Hausdorff spaces.

\begin{enumerate}
    \item \label{item:pushforward-compose-subset}
    
    One has $ \mathcal{ M } ( p^{ 2 , 3 } \circ p^{ 1 , 2 } , X_1 ) \subset\mathcal{ M } ( p^{ 1 , 2 } , X_1 ) $.

\item \label{item:pushforward-compose-image}

One has $ p_*^{ 1 , 2 } ( \mathcal{ M } ( p^{ 2 , 3 } \circ p^{ 1 , 2 } , X_1 ) ) \subset \mathcal{ M } ( p^{ 2 , 3 } , X_2 ) $.

\item \label{item:pushforward-compose-compose}

One has $ ( p^{ 2 , 3 } \circ p^{ 1 , 2 } )_* = p_*^{ 2 , 3 } \circ p_*^{ 1 , 2 } $.
    
\end{enumerate}
    
\end{lemma}

\begin{proof}

\begin{enumerate}
    \item 

Let $ K \subset X_2 $ be a compact set and $ \mu \in \mathcal{ M } ( p^{ 2 , 3 } \circ p^{ 1 , 2 } , X_1 ) $.
Then it suffices to show that $ \mathrm{ supp } ( \mu ) \cap ( p^{ 1 , 2 } )^{ - 1 } ( K ) $ is compact.
We have
\begin{align}
    \mathrm{ supp } ( \mu ) \cap ( p^{ 1 , 2 } )^{ - 1 } ( K )
    \subset \mathrm{ supp } ( \mu ) \cap ( p^{ 2 , 3 } \circ p^{ 1 , 2 } )^{ - 1 } ( p^{ 2 , 3 } ( K ) ).
\end{align}
Since $ p^{ 2 , 3 } ( K ) $ is compact, $ \mathrm{ supp } ( \mu ) \cap ( p^{ 2 , 3 } \circ p^{ 1 , 2 } )^{ - 1 } ( p^{ 2 , 3 } ( K ) ) $ is also compact by $ \mu \in \mathcal{ M } ( p^{ 2 , 3 } \circ p^{ 1 , 2 } , X_1 ) $.
In addition, the sets $ \mathrm{ supp } ( \mu ) $ and $ ( p^{ 1 , 2 } )^{ - 1 } ( K ) $ are closed and hence $ \mathrm{ supp } ( \mu ) \cap ( p^{ 1 , 2 } )^{ - 1 } ( K ) $ is compact.

\item

Let $ \mu \in \mathcal{ M } ( p^{ 2 , 3 } \circ p^{ 1 , 2 } , X_1 ) $ and $ K \subset X_3 $ be a compact set.
It suffices to show that $ ( p^{ 2 , 3 } )^{ - 1 } ( K ) \cap \mathrm{ supp } ( p_*^{ 1 , 2 } ( \mu ) ) $ is compact.
Then
\begin{align}
    ( p^{ 2 , 3 } )^{ - 1 } ( K ) \cap \mathrm{ supp } ( p_*^{ 1 , 2 } ( \mu ) )
    & \subset ( p^{ 2 , 3 } )^{ - 1 } ( K ) \cap p^{ 1 , 2 } ( \mathrm{ supp } ( \mu ) ) \\
    & = p^{ 1 , 2 } ( ( p^{ 2 , 3 } \circ p^{ 1 , 2 } )^{ - 1 } ( K ) \cap \mathrm{ supp } ( \mu ) )
\end{align}
holds by Lemma \ref{lem:pushforward-support}.
Since $ ( p^{ 2 , 3 } \circ p^{ 1 , 2 } )^{ - 1 } ( K ) \cap \mathrm{ supp } ( \mu ) $ is compact by $ \mu \in \mathcal{ M } ( p^{ 2 , 3 } \circ p^{ 1 , 2 } , X_1 ) $, the set $ ( p^{ 2 , 3 } )^{ - 1 } ( K ) \cap \mathrm{ supp } ( p_*^{ 1 , 2 } ( \mu ) ) $ is also compact.

\item 
Let $ \mu \in \mathcal{ M } ( p^{ 2 , 3 } \circ p^{ 1 , 2 } , X_1 ) $ and $ \alpha_3 \in C_c ( X_3 ) $.
Then it suffices to show
\begin{align}
   \int_{ X_3 } \alpha_3 \; d ( p^{ 2 , 3 } \circ p^{ 1 , 2 } )_* ( \mu )
   = \int_{ X_3 } \alpha_3 \; d p_*^{ 2 , 3 } \circ p_*^{ 1 , 2 } ( \mu ).
\end{align}
We note that both sides of this equality can be defined by \ref{item:pushforward-compose-subset} and \ref{item:pushforward-compose-image}.
Then there exists $ \alpha_1 \in C_c ( X_1 ) $ such that $ \alpha_1 ( x_1 ) = \alpha_3 \circ p^{ 2 , 3 } \circ p^{ 1, 2 } ( x_1 ) $ for any $ x_1 \in \mathrm{ supp } ( \mu ) $.
For any $ x_1 , x_1' \in \mathrm{ supp } ( \mu ) $ with $ p^{ 1, 2 } ( x_1 ) = p^{ 1, 2 } ( x_1' ) $, we have
\begin{align}
\alpha_1 ( x_1 )
   = \alpha_3 \circ p^{ 2 , 3 } \circ p^{ 1, 2 } ( x_1 )
   = \alpha_3 \circ p^{ 2 , 3 } \circ p^{ 1, 2 } ( x_1' )
   = \alpha_1 ( x_1' ).
\end{align}
Thus, there exists $ \alpha_2 \in C ( p^{ 1 , 2 } ( \mathrm{ supp } ( \mu ) ) ) $ such that $ \alpha_2 \circ p^{ 1, 2 } ( x_1 ) = \alpha_1 ( x_1 ) $ for any $ x_1 \in \mathrm{ supp } ( \mu ) $.
We can extend $ \alpha_2 $ to $ X_2 \to \mathbb{ C } $ such that $ \alpha_2 \in C_c ( X_2 ) $ by the compactness of $ \mathrm{ supp } ( \alpha_1 ) $ and the Tietze extension theorem.
For any $ x_2 \in \mathrm{ supp } ( p_*^{ 1 , 2 } ( \mu ) ) $, there exists $ x_1 \in \mathrm{ supp } ( \mu ) $ such that $ x_2 = p^{ 1 , 2 } ( x_1 ) $ by Lemma \ref{lem:pushforward-support}.
Then
\begin{align}
    \alpha_3 \circ p^{ 2 , 3 } ( x_2 )
    = \alpha_3 \circ p^{ 2 , 3 } \circ p^{ 1 , 2 } ( x_1 )
    = \alpha_1 ( x_1 )
    = \alpha_2 \circ p^{ 1 , 2 } ( x_1 )
    = \alpha_2 ( x_2 )
\end{align}
holds and hence
\begin{align}
    \int_{ X_3 } \alpha_3 \; d p_*^{ 2 , 3 } \circ p_*^{ 1 , 2 } ( \mu )
    = \int_{ X_2 } \alpha_2 \; d p_*^{ 1 , 2 } ( \mu )
    = \int_{ X_1 } \alpha_1 \; d \mu
    = \int_{ X_3 } \alpha_3 \; d ( p^{ 2 , 3 } \circ p^{ 1 , 2 } )_* ( \mu )
\end{align}
is obtained.
\qedhere

\end{enumerate}
    
\end{proof}

\subsection{Application to homogeneous spaces}
\label{subsec:pushforward-apply}

By applying the properties of pushforward measures in Subsection \ref{subsec:pushforward-definition} to homogeneous spaces, we can prove Theorem \ref{thm:main}.
For this purpose, we use the following lemma:

\begin{lemma}
    \label{lem:MH-represent}

Let $ dg $ be a left Haar measure of a locally compact group $ G $.

\begin{enumerate}
\item \label{item:MH-represent-support}
One has $ \mathrm{ supp } ( \phi ( g ) dg ) = \mathrm{ supp } ( \phi ) $ for any $ \phi \in C ( G ) $.
    
    \item \label{item:MH-represent-explicit}

Let $ H < G $ and $ p^{ G \to G / H } $ be the canonical map.
Then
\begin{align}
    M_H ( G ) = \mathcal{ M } ( p^{ G \to G / H } , G ) \cap M_{ \{ e \} } ( G ).
\end{align}
In particular, the pushforward $ p_*^{ G \to G / H } ( \mu ) $ can be defined for any $ \mu \in M_H ( G ) $.

    \item \label{item:MH-represent-subset}

    Let $ N $, $ \overline{ G } $, and $ p^{ G \to G / H } $ be as in Theorem \ref{thm:main} (which does not necessarily satisfy $ N \lhd G $).
    Then $ \mathcal{ M } ( p^{ G \to G / H } , G ) \subset \mathcal{ M } ( p^{ G \to \overline{ G } } , G ) $.

\end{enumerate}

\end{lemma}

\begin{proof}

\begin{enumerate}
    \item 

The conclusion is obtained by $ \mathrm{ supp } ( dg ) = G $.

\item

The conclusion is obtained by \ref{item:MH-represent-support} and Example \ref{ex:MH-H-trivial}.

\item

Since $ p^{ G \to G / H } = p^{ \overline{ G } \to G / H } \circ p^{ G \to \overline{ G } } $, we obtain $ \mathcal{ M } ( p^{ G \to G / H } , G ) \subset \mathcal{ M } ( p^{ G \to \overline{ G } } , G ) $ by Lemma \ref{lem:pushforward-compose} \ref{item:pushforward-compose-subset}. 
\qedhere
    
\end{enumerate}
    
\end{proof}

A left Haar measure of quotient groups can be given as the following fact:

\begin{fact}[{\cite[Theorem 2.51]{MR3444405}}]
    \label{fact:quotient-Haar}

For $ G $, $ N $, and $ \overline{ G } $ as in Corollary \ref{cor:modular}, let $ dg $ and $ dn $ be left Haar measures of $ G $ and $ N $, respectively. 
Then there exists a left Haar measure $ d \overline{ g } $ of $ \overline{ G } $ such that
\begin{align}
    \int_{ \overline{ G } } \int_{ N } \beta ( g n ) dn d \overline{ g }
    = \int_{ G } \beta ( g ) dg
\end{align}
for any $ \beta \in C_c ( G ) $, where $ \overline{ g } \coloneqq p^{ G \to \overline{ G } } ( g ) $.
    
\end{fact}

By using a left Haar measure $ d \overline{ g } $ as in Fact \ref{fact:quotient-Haar}, the pushforward $ p_*^{ G \to \overline{ G } } ( \mu ) $ (which is well-defined by Lemma \ref{lem:MH-represent} \ref{item:MH-represent-explicit}) can be written explicitly for any $ \mu \in M_N ( G ) $ as follows:

\begin{lemma}
    \label{lem:pushforward-represent}

Let $ G $, $ N $, $ \overline{ G } $, $ dg $, $ dn $, and $ d \overline{ g } $ be as in Fact \ref{fact:quotient-Haar}.
Then
\begin{align}
    p_*^{ G \to \overline{ G } } ( \phi ( g ) dg )
    = \int_{ N } \phi ( g n ) dn d \overline{ g } \label{eq:pushforward-represent-state}
\end{align}
for any $\phi \in C ( G ) $ for which $ p^{ G \to \overline{ G } }|_{ \mathrm{ supp } ( \phi ) } $ is proper.
In particular, one has 
\begin{align}
    p_*^{ G \to \overline{ G } } ( M_N ( G ) ) \subset M_{ \overline{ N } } ( \overline{ G } ). \label{eq:pushforward-represent-subset}
\end{align}
 
\end{lemma}

\begin{proof}

Since
\begin{align}
    \int_{ \overline{ G } } \alpha ( \overline{ g } ) dp_*^{ G \to \overline{ G } } ( \phi ( g ) dg )
    & = \int_{ G } \alpha ( \overline{ g } ) \phi ( g ) dg \\
    & = \int_{ \overline{ G } } \int_{ N } \alpha ( \overline{ g n } ) \phi ( g n ) dn d \overline{ g } \\
    & = \int_{ \overline{ G } } \int_{ N } \alpha ( \overline{ g } ) \phi ( g n ) dn d \overline{ g } \\
    & = \int_{ \overline{ G } } \alpha ( \overline{ g } ) \int_{ N } \phi ( g n ) dn d \overline{ g }
\end{align}
holds for any $ \alpha \in C_c ( \overline{ G } ) $ by Fact \ref{fact:quotient-Haar}, we obtain \eqref{eq:pushforward-represent-state}.
In particular, we have
\begin{align}
    \int_{ N } \phi ( g n ) dn d \overline{ g } \in M_{ \overline{ N } } ( \overline{ G } )
\end{align}
and hence \eqref{eq:pushforward-represent-subset} is obtained.
\end{proof}

Now we prove Theorem \ref{thm:main} by using Lemma \ref{lem:MH-represent} and Lemma \ref{lem:pushforward-represent}.

\begin{proof}[Proof of Theorem \ref{thm:main}]

\begin{enumerate}
    \item 
    
    By Lemma \ref{lem:MH-represent} \ref{item:MH-represent-explicit} and \ref{item:MH-represent-subset}, we obtain
\begin{align}
    M_H ( G )
    = \mathcal{ M } ( p^{ G \to G / H } , G ) \cap M_{ \{ e \} } ( G )
    \subset \mathcal{ M } ( p^{ G \to \overline{ G } } , G ) \cap M_{ \{ e \} } ( G )
    = M_N ( G ).
\end{align}

\item

We have
\begin{align}
    \mu \in M_H ( G )
    \subset \mathcal{ M } ( p^{ G \to G / H } , G )
    = \mathcal{ M } ( p^{ \overline{ G } \to G / H } \circ p^{ G \to \overline{ G } } , G )
\end{align}
by Lemma \ref{lem:MH-represent} \ref{item:MH-represent-explicit}.
Thus, the pushforwards $ ( p^{ \overline{ G } \to G / H } \circ p^{ G \to \overline{ G } } )_* ( \mu ) $ and $ p_*^{ \overline{ G } \to G / H } \circ p_*^{ G \to \overline{ G } } ( \mu ) $ can be defined and we obtain \eqref{eq:main-pushforward-well-defined-equal} by Lemma \ref{lem:pushforward-compose}.

\item

Since
\begin{align}
    \mu
    & \in M_H ( G ) \\
    & \subset \mathcal{ M } ( p^{ G \to G / H } , G ) \\
    & = \mathcal{ M } ( ( p^{ G / H \to \overline{ G } / \overline{ H } } )^{ - 1 } \circ p^{ \overline{ G } \to \overline{ G } / \overline{ H } } \circ p^{ G \to \overline{ G } } , G ) \\
    & \subset \mathcal{ M } ( p^{ \overline{ G } \to \overline{ G } / \overline{ H } } \circ p^{ G \to \overline{ G } } , G )
\end{align}
by Lemma \ref{lem:pushforward-compose} \ref{item:pushforward-compose-subset} and Lemma \ref{lem:MH-represent} \ref{item:MH-represent-explicit}, the pushforward $ p_*^{ \overline{ G } \to \overline{ G } / \overline{ H } } \circ p_*^{ G \to \overline{ G } } ( \mu ) $ can be defined.
Thus, it follows from Lemma \ref{lem:pushforward-compose} and $ p^{ \overline{ G } \to \overline{ G } / \overline{ H } } \circ p^{ G \to \overline{ G } } = p^{ G / H \to \overline{ G } / \overline{ H } } \circ p^{ G \to G / H } $ that $ p_*^{ G / H \to \overline{ G } / \overline{ H } } \circ p_*^{ G \to G / H } ( \mu ) $ can also be defined, and \eqref{eq:main-normal-commutative-state} holds.

\item

It follows from \ref{item:main-pushforward-subset} and Lemma \ref{lem:pushforward-represent} that
\begin{align}
    p_*^{ G \to \overline{ G } } ( \mu )
    \in p_*^{ G \to \overline{ G } } ( M_{ H } ( G ) )
    \subset p_*^{ G \to \overline{ G } } ( M_{ N } ( G ) )
    \subset M_{ \overline{ N } } ( \overline{ G } ).
\end{align}
Thus, we obtain
\begin{align}
    p_*^{ G \to \overline{ G } } ( \mu )
    \in \mathcal{ M } ( p^{ \overline{ G } \to \overline{ G } / \overline{ H } } , G ) \cap M_{ \overline{ N } } ( \overline{ G } )
    = M_{ \overline{ H } } ( \overline{ G } )
\end{align}
by \ref{item:main-normal-commutative} and Lemma \ref{lem:MH-represent} \ref{item:MH-represent-explicit}.
\qedhere

\end{enumerate}
    
\end{proof}

\section{Proof of Theorem \ref{thm:quotient-pushforward} and Corollary \ref{cor:modular}}
\label{sec:proof-modular}

In this section, we prove Theorem \ref{thm:quotient-pushforward} and Corollary \ref{cor:modular}.

\begin{proof}[Proof of Theorem \ref{thm:quotient-pushforward}]

We have
\begin{align}
    \int_{ G / H } \int_{ H } \nu ( g h ) dh d p_*^{ G \to G / H } ( \phi ( g ) dg )
    & = \int_{ G } \int_{ H } \nu ( g h ) dh \phi ( g ) dg \\
    & = \int_{ H } \int_{ G } \nu ( g h ) \phi ( g ) dg dh \\
    & = \int_{ H } \int_{ G } \frac{ \nu ( g ) \phi ( g h^{ - 1 } ) }{ \Delta_G ( h ) } dg dh \\
    & = \int_{ G }  \nu ( g ) \int_{ H } \frac{ \phi ( g h^{ - 1 } ) }{ \Delta_G ( h ) } dh dg
\end{align}
by Fubini's theorem.
Since
\begin{align}
    \int_{ H } \frac{ \phi ( g h^{ - 1 } ) }{ \Delta_G ( h ) } dh
    = \int_{ H } \phi ( g h ) \frac{ \Delta_G ( h ) }{ \Delta_H ( h ) } dh
\end{align}
holds, we obtain \eqref{eq:quotient-pushforward-state}.
\end{proof}

We utilize Theorem \ref{thm:main} and Theorem \ref{thm:quotient-pushforward} together with the following fact to prove Corollary \ref{cor:modular}:

\begin{fact}[{\cite[Proposition 2.50]{MR3444405}}]
\label{fact:quotient-surjective}

Let $ G $, $ H $, $ p^{ G \to G / H } $, $ dh $ be as in Theorem \ref{thm:quotient-pushforward}, and $ \alpha \in C_c ( G / H ) $.
Then there exists $ \nu \in C_c ( G ) $ such that
\begin{align}
   \alpha ( g H ) = \int_H \nu ( g h ) dh
\end{align}
holds for any $ g \in G $.
    
\end{fact}

Then we complete the proof of Corollary \ref{cor:modular} by using Theorem \ref{thm:main}, Theorem \ref{thm:quotient-pushforward}, and Fact \ref{fact:quotient-surjective}.

\begin{proof}[Proof of Corollary \ref{cor:modular}]

Suppose $ \phi \in C ( G ) $ satisfies $ p^{ G \to G / H }|_{ \mathrm{ supp } ( \phi ) } $ is proper.
Now we show
\begin{align}
    p_*^{ \overline{ G } \to \overline{ G } / \overline{ H } } \circ p_*^{ G \to \overline{ G } } ( \phi ( g h' ) dg )
    = \frac{ \Delta_{ \overline{ H } } ( \overline{ h' } ) p_*^{ \overline{ G } \to \overline{ G } / \overline{ H } } \circ p_*^{ G \to \overline{ G } } ( \phi ( g ) dg ) }{ \Delta_{ \overline{ G } } ( \overline{ h' } ) \Delta_H ( h' ) }. \label{eq:modular-right-action}
\end{align}
Then
\begin{align}
    p_*^{ G \to \overline{ G } } ( \phi ( g h' ) dg )
    = \int_N \phi ( g n h' ) dn d \overline{ g } \label{eq:Fact-pushforward-apply}
\end{align}
holds by Lemma \ref{lem:pushforward-represent}.
Now we let $ \alpha \in C_c ( \overline{ G } / \overline{ H } ) $.
By Fact \ref{fact:quotient-surjective}, there exists $ \nu \in C_c ( \overline{ G } ) $ such that
\begin{align}
   \alpha ( \overline{ g H } ) = \int_{ \overline{ H } } \nu ( \overline{ g h } ) d \overline{ h }
\end{align}
holds for any $ g \in G $, where $ \overline{ g H } \coloneqq p^{ G / H \to \overline{ G } / \overline{ H } } ( g H ) $.
Then
\begin{align}
   \int_{ \overline{ G } / \overline{ H } } \alpha ( \overline{ g H } ) dp_*^{ \overline{ G } \to \overline{ G } / \overline{ H } } \circ p_*^{ G \to \overline{ G } } ( \phi ( g h' ) dg )
    & = \int_{ \overline{ G } / \overline{ H } } \int_{ \overline{ H } } \nu ( \overline{ g h } ) d \overline{ h } \; dp_*^{ \overline{ G } \to \overline{ G } / \overline{ H } } \left( \int_N \phi ( g n h' ) dn d \overline{ g } \right) \\
    & = \int_{ \overline{ G } } \nu ( \overline{ g } ) \int_{ \overline{ H } } \int_N \phi ( g h n h' ) dn \frac{ \Delta_{ \overline{ G } } ( \overline{ h } ) }{ \Delta_{ \overline{ H } } ( \overline{ h } ) } d \overline{ h } d \overline{ g }
\end{align}
holds by \eqref{eq:Fact-pushforward-apply} and Theorem \ref{thm:quotient-pushforward}.
Since
\begin{align}
    \int_{ \overline{ H } } \int_N \phi ( g h n h' ) dn \frac{ \Delta_{ \overline{ G } } ( \overline{ h } ) }{ \Delta_{ \overline{ H } } ( \overline{ h } ) } d \overline{ h }
    & = \int_{ \overline{ H } } \int_N \phi ( g h n h' ) \frac{ \Delta_{ \overline{ G } } ( \overline{ h n } ) }{ \Delta_{ \overline{ H } } ( \overline{ h n } ) } dn d \overline{ h } \\
    & = \int_{ H } \phi ( g h h' ) \frac{ \Delta_{ \overline{ G } } ( \overline{ h } ) }{ \Delta_{ \overline{ H } } ( \overline{ h } ) } dh \\
    & = \int_{ H } \phi ( g h ) \frac{ \Delta_{ \overline{ G } } ( \overline{ h h'^{ - 1 } } ) }{ \Delta_{ \overline{ H } } ( \overline{ h h'^{ - 1 } } ) } dh \frac{ 1 }{ \Delta_H ( h' ) } \\
    & = \int_{ H } \phi ( g h ) \frac{ \Delta_{ \overline{ G } } ( \overline{ h } ) }{ \Delta_{ \overline{ H } } ( \overline{ h } ) } dh \frac{ \Delta_{ \overline{ H } } ( h' ) }{ \Delta_{ \overline{ G } } ( h' ) \Delta_H ( h' ) }
\end{align}
holds by Fact \ref{fact:quotient-Haar}, we have
\begin{align}
    \phantom{ = } & \int_{ \overline{ G } / \overline{ H } } \alpha ( \overline{ g H } ) dp_*^{ \overline{ G } \to \overline{ G } / \overline{ H } } \circ p_*^{ G \to \overline{ G } } ( \phi ( g h' ) dg ) \\
    = & \int_{ \overline{ G } / \overline{ H } } \alpha ( \overline{ g H } ) dp_*^{ \overline{ G } \to \overline{ G } / \overline{ H } } \circ p_*^{ G \to \overline{ G } } ( \phi ( g ) dg ) \frac{ \Delta_{ \overline{ H } } ( h' ) }{ \Delta_{ \overline{ G } } ( h' ) \Delta_H ( h' ) }
\end{align}
and hence \eqref{eq:modular-right-action}.

In particular, we have
\begin{align}
    p_*^{ G \to G / H } ( \phi ( g h' ) dg )
    = \frac{ \Delta_{ H } ( h' ) p_*^{ G \to G / H } ( \phi ( g ) dg ) }{ \Delta_{ G } ( h' ) \Delta_H ( h' ) }
    = \frac{ p_*^{ G \to G / H } ( \phi ( g ) dg ) }{ \Delta_{ G } ( h' ) }
\end{align}
by applying $ N = \{ e \} $ to \eqref{eq:modular-right-action}.
Thus, it follows from \eqref{eq:modular-right-action} and Theorem \ref{thm:main} \ref{item:main-normal-commutative} that
\begin{align}
    \frac{ \Delta_{ \overline{ H } } ( \overline{ h' } ) p_*^{ \overline{ G } \to \overline{ G } / \overline{ H } } \circ p_*^{ G \to \overline{ G } } ( \phi ( g ) dg ) }{ \Delta_{ \overline{ G } } ( \overline{ h' } ) \Delta_H ( h' ) }
    & = p_*^{ \overline{ G } \to \overline{ G } / \overline{ H } } \circ p_*^{ G \to \overline{ G } } ( \phi ( g h' ) dg ) \\
    & = p_*^{ G / H \to \overline{ G } / \overline{ H } } \circ p_*^{ G \to G / H } ( \phi ( g h' ) dg ) \\
    & = p_*^{ G / H \to \overline{ G } / \overline{ H } } \left( \frac{ p_*^{ G \to G / H } ( \phi ( g ) dg ) }{ \Delta_{ G } ( h' ) } \right) \\
    & = \frac{ p_*^{ G / H \to \overline{ G } / \overline{ H } } \circ p_*^{ G \to G / H } ( \phi ( g ) dg ) }{ \Delta_{ G } ( h' ) } \\
    & = \frac{ p_*^{ \overline{ G } \to \overline{ G } / \overline{ H } } \circ p_*^{ G \to \overline{ G } } ( \phi ( g ) dg ) }{ \Delta_{ G } ( h' ) }.
\end{align}
Therefore, we obtain \eqref{eq:modular-state}.
\end{proof}

\section*{Acknowledgement}

This work was supported by the RIKEN Special Postdoctoral Researcher (SPDR) program.
The author would like to express his sincere gratitude to Prof.~Toshiyuki Kobayashi for repeated discussions and helpful advice for this paper.
The author is also grateful to Yuichiro Tanaka, Toshihisa Kubo, and the anonymous referee for their careful comments.

\printbibliography

\noindent
Takashi Satomi: RIKEN Interdisciplinary Theoretical and Mathematical Sciences (iTHEMS), Wako, Saitama 351-0198, Japan.

\noindent
E-mail: takashi.satomi@riken.jp

\end{document}